\documentclass[11pt]{amsart}
\usepackage{amsmath,amsthm,amssymb,amscd,url}
\usepackage{amsmath,amsthm,MnSymbol,amsfonts}
\usepackage{verbatim}
\usepackage[all]{xy} \SelectTips {cm}{}
\usepackage{tikz}
\usetikzlibrary{matrix,arrows,decorations.pathreplacing,decorations.pathmorphing}
\usepackage{amsmath,amsthm,amssymb,amscd,url}

\usepackage{verbatim}
\usepackage[all]{xy} \SelectTips {cm}{}
\usepackage{tikz}
\usetikzlibrary{matrix,arrows,decorations.pathreplacing,decorations.pathmorphing}

\newtheorem{theorem}{Theorem}[section] 
\newtheorem{lemma}[theorem]{Lemma}     
\newtheorem{corollary}[theorem]{Corollary}

\newtheorem{proposition}[theorem]{Proposition}

\newtheorem{example}[theorem]{Example}
\theoremstyle{remark}
\newtheorem{remark}[theorem]{Remark}
\theoremstyle{definition}
\newtheorem{definition}[theorem]{Definition}
\usepackage{amsfonts,url}

\newcommand{\q}{/\!/}

\def\Irr{mathbf{Irr}}

\renewcommand{\k}{\ensuremath{k}}
\newcommand{\CC}{\ensuremath{\mathbb{C}}}

\newcommand{\OOO}{\ensuremath{\mathcal{O}}}

\newcommand{\Hom}{\ensuremath{\mathrm{Hom}}}
\def\Irr{\mathrm{Irr}}

\def\op{\mathrm{op}}
\def\Spec{\mathrm{Spec}}
\def\maxSpec{\mathrm{maxSpec}}

\def\Mat{\mathrm{M}}
\newcommand{\Prim}{\ensuremath{\mathrm{Prim}}}

\def\fE{\mathfrak{E}}

\DeclareMathOperator{\HP}{HP}

\newcommand{\Z}{\mathbb{Z}}

\newcommand{\C}{\mathbb{C}}

\newcommand{\cH}{\mathcal{H}}

\newcommand{\cO}{\mathcal{O}}
\newcommand\cI{\mathcal{I}}

\newcommand{\SL}{\mathrm{SL}}

\newcommand{\GL}{\mathrm{GL}}

\begin{document}
\title{Morita equivalence for $k$-algebras}

\begin{abstract}   We review Morita equivalence for finite type $k$-algebras $A$ and also a weakening of Morita equivalence which we call 
\emph{stratified equivalence}.   
The spectrum of $A$ is the set of equivalence classes of irreducible $A$-modules.  For any finite type $k$-algebra $A$, the spectrum  of $A$ is in bijection with the set of primitive ideals of $A$. 
The stratified equivalence relation preserves the spectrum of $A$
 and also preserves the periodic cyclic homology of $A$. However, the stratified equivalence relation
permits a tearing apart of strata in the primitive ideal space which is not allowed by Morita equivalence.
A key example illustrating the distinction between Morita equivalence and stratified equivalence  is provided by affine Hecke algebras
associated to affine Weyl groups.
\end{abstract}

\author[A.-M. Aubert]{Anne-Marie Aubert}
\address{CNRS, Sorbonne Universit\'e, Universit\'e paris Diderot, Institut de Math\'ematiques de Jussieu -- Paris Rive Gauche, IMJ-PRG, F-75005 Paris, France}
\email{anne-marie.aubert@imj-prg.fr}
\author[P. Baum]{Paul Baum}
\address{Mathematics Department, Pennsylvania State University, University Park, PA 16802, USA}
\email{pxb6@psu.edu}
\author[R. Plymen]{Roger Plymen}
\address{
School of Mathematics, Southampton University, Southampton SO17 1BJ,  England 
\emph{and} 
School of Mathematics, Manchester University, Manchester M13 9PL, England}
\email{
r.j.plymen@soton.ac.uk \quad 
roger.j.plymen@manchester.ac.uk}
\author[M. Solleveld]{Maarten Solleveld}
\address{IMAPP, Radboud Universiteit Nijmegen, Heyendaalseweg 135, 
6525AJ Nijmegen, the Netherlands}
\email{m.solleveld@science.ru.nl}

\date{\today}
\maketitle

\tableofcontents

\section{Introduction}   The subject of Morita equivalence begins with the classical paper of Morita \cite{Mor}.   Morita equivalence of rings is  well-established and  widely used \cite[\S18--19]{Lam}.  
   
Now let $k$ be the coordinate algebra of a complex affine variety.   Equivalently, $k$ is a unital algebra over the complex numbers which 
is commutative, finitely generated, and nilpotent-free. Morita equivalence for $k$-algebras, the starting-point  of this article,  is a more recent development, which introduces several nuances 
related to the $k$-module structure of the algebras.

  A $k$-algebra is an algebra $A$ over the complex numbers $\mathbb{C}$ which is a $k$-module
(with an evident compatibility between the algebra structure of $A$ and the $k$-module structure of $A$). $A$ is not required to have a unit. $A$ is not required
to be commutative.
A $k$-algebra $A$ is of finite type if as a $k$-module
$A$ is finitely generated.   The spectrum of $A$ is, by definition, the set of equivalence classes of irreducible $A$-modules. For any finite type $k$-algebra $A$, the spectrum  of $A$ is in bijection with the set of primitive ideals of $A$.

The set of primitive ideals of $A$ will be denoted $\Prim(A)$.   When $A$ is unital and commutative, let $\maxSpec(A)$ denote the set of maximal ideals in $A$, and let $\Spec(A)$ denote  
the set of prime ideals in $A$.      In particular, let $A = k$.   Then we have
\[
\Prim(A) \simeq \maxSpec(A) \subset \Spec(A).
\]
The inclusion is strict, for the prime spectrum $\Spec(A)$ contains a generic point, namely the zero ideal $\mathbf{0}$ in $A$.   The set $\Prim(A)$ 
consists of the $\C$-rational points of the affine variety underlying $k$.

So $\Prim(A)$ is a generalisation, for $k$-algebras $A$ of finite type,  of the maximal ideal spectrum of a unital commutative ring,
rather than the prime spectrum of a unital commutative ring.  For a finite type $k$-algebra $A$, the zero ideal $\mathbf{0}$ is a primitive ideal if and only if
$k = \C$ and $A = M_n(\C)$ the algebra of all $n$ by $n$ matrices with entries in $\C$.
Quite generally, any irreducible representation of a finite type $k$-algebra $A$ is a surjection of $A$ onto $M_n(\C)$ for some $n$.   This implies that any primitive 
ideal is maximal.  

In some situations, Morita equivalence can be too strong and we are led to introduce a  weakening of this concept, which we call  \emph{stratified equivalence}.

The stratified equivalence relation preserves the spectrum of $A$
 and also preserves the periodic cyclic homology of $A$. In addition, the stratified equivalence relation
permits a tearing apart of \emph{strata} in the primitive ideal space which is not allowed by Morita equivalence.

 Denote by $\Gamma$ a finite group.   In \S 10, we show that stratified equivalence persists under the formation of tensor products 
$A \otimes B$, and the 
formation of crossed products $A \rtimes \Gamma$.

A key example illustrating the distinction between Morita equivalence and stratified equivalence  is provided by affine Hecke algebras.  Let $(X,Y,R,R^\vee)$ be a root 
datum in the standard sense \cite{L}.   This root datum delivers the following items: 

-- a finite Weyl group $W_0$

--  an extended affine Weyl group $W_0X: = W_0 \ltimes X$

-- for each $q \in \C^\times$, an affine Hecke algebra $\cH_q$

-- a complex torus $T: = \Hom(X, \C^\times)$   

-- a complex variety $X: = T/W_0$

-- a canonical isomorphism $\cO(X) \simeq Z(\cH_q)$
\medskip

Set $k: = \cO(X)$.   Then, for all $q \in\mathbb{C}^{\times}$, $\mathcal{H}_q$ is a unital finite type $k$-algebra.   
If $q = 1$, then $\cH_1$ is the group algebra of the extended affine Weyl group:
\[
\cH_1 = \C[W_0X]
\]

\begin{theorem}   Except for $q$ in a finite set of roots of
unity, none of which is 1,  $\cH_q$ and $\cH_1$ are stratified equivalent as $k$-algebras.   If $q \neq1$ then $\cH_q$ and $\cH_1$ are not Morita equivalent as $k$-algebras.   
\end{theorem}

A finite type $k$-algebra can be viewed as a noncommutative complex affine variety. So the setting of this article can be viewed as noncommutative algebraic geometry.   

We conclude this article with a detailed account of the affine Hecke algebra $\cH_q$ attached to the Lie group $\SL_2(\C)$.   

For an application of stratified equivalence to the representation theory of $p$-adic groups, see \cite{ABPS}.

\section{$k$-algebras}  If $X$ is an affine algebraic variety over the complex numbers $\CC$, then $\mathcal{O}(X)$ will denote the coordinate algebra of $X$.
 Set  $\k=\mathcal{O}(X)$.   Equivalently, $k$ is a unital algebra over the complex numbers which 
is unital, commutative, finitely generated, and nilpotent-free. The Hilbert Nullstellensatz implies that there is an equivalence of categories 
\vspace{3mm}\\
\begin{center}
\begin{tikzpicture}[node distance=.5cm]
\node[text width=4.5cm, left delimiter=(, right delimiter=),left] at (-.6,0)
{unital commutative finitely generated nilpotent-free \mbox{$\CC$-algebras}};
\node at (0,0){$\sim$};
\node[text width=3.0cm,left delimiter=(, right delimiter=),right] at (.6,0)
{affine complex algebraic varieties};
\node at (4.4,.5) {$\op$};
\end{tikzpicture}\\
$\mathcal{O}(X)\mapsto X$
\end{center}
\vspace{4mm}
Here ``op'' denotes the opposite category.

\begin{definition}
A \emph{$\k$-algebra} is a $\CC$-algebra $A$ such that $A$ is a unital (left) $\k$-module with:
\[
\lambda(\omega a)=\omega(\lambda a)=(\lambda\omega)a\quad \forall (\lambda,\omega,a)\in\CC \times \k \times A
\]
and
\[
\omega(a_1a_2)=(\omega a_1)a_2=a_1(\omega a_2)\quad \forall (\omega,a_1,a_2)\in \k\times A\times A.
\] \vspace{2mm}
\end{definition}
\begin{remark} $A$ is not required to have a unit.
\end{remark} 

The centre of $A$ will be denoted $Z(A)$:
\[
Z(A):=\{c\in A\mid ca=ac\,\,\forall a\in A\}.
\]

\begin{remark}
Let $A$ be a unital $\k$-algebra. Denote the unit of $A$ by $1_A$.\\
The map $\omega\mapsto \omega.1_A$  is then a unital morphism of $\CC$-algebras\\
\[
k\longrightarrow Z(A)
\]
i.e. unital $\k$-algebra = unital $\CC$-algebra $A$ with a given unital morphism of $\CC$-algebras
\[
\k\longrightarrow Z(A)
\]

If $A$ does not have a unit, a $k$-structure is equivalent to a unital morphism of $\C$-algebras
\[
k \to Z(\mathcal{M}A)
\]
where $\mathcal{M} A$ is the multiplier algebra of $A$.   

\end{remark}

\begin{definition}
Let $A,B$ be two $\k$-algebras. A \emph{morphism of \k-algebras} is a morphism of $\CC$-algebras
\[
f\colon A\to B
\]
which is also a morphism of (left) $\k$-modules,
\[
f(\omega a)=\omega f(a)\quad \forall (\omega,a)\in \k\times A.
\]
\end{definition}
\begin{definition}
Let $A$ be a $\k$-algebra. A \emph{representation} of $A$ [or a (left) \emph{$A$-module}] is a $\CC$-vector space $V$ with given morphisms of $\CC$-algebras
\[
A\longrightarrow \mathrm{Hom}_\CC(V,V)
\]
\[
\k\longrightarrow \mathrm{Hom}_\CC(V,V)
 \] 
such that 
\begin{enumerate}
\item $\k\longrightarrow \mathrm{Hom}_\CC(V,V)$ is unital.
\item $(\omega a)v=\omega(av)=a(\omega v)$ \;$\forall (\omega,a,v)\in k\times A\times V$.
\end{enumerate}
\end{definition}
\noindent\emph{From now on in this article, $A$ will denote a $k$-algebra.}    
\bigskip

\noindent A representation of $A$ 
\[
A\longrightarrow \mathrm{Hom}_\CC(V,V)
\]
\[
\k\longrightarrow \mathrm{Hom}_\CC(V,V)
 \] 
will often be denoted 
\[
A\to\mathrm{Hom}_{\mathbb{C}}(V, V)
\]
it being understood that the action of $k$ on $V$
\[
k\to\mathrm{Hom}_{\mathbb{C}}(V, V)
\]
is part of the given structure.

\begin{definition}
A representation $\varphi\colon A\to \mathrm{Hom}_\CC(V,V)$ is 
\emph{non-degenerate} if $AV=V$, i.e. for any $v\in V$, there are 
$v_1$, $v_2$, $\ldots$, $v_r$ in $V$ and $a_1$, $a_2$, $\ldots$, $a_r$ in $A$ with 
\[
v=a_1v_1+a_2v_2+\cdots+a_rv_r.\vspace{3mm}
\]
\end{definition}
\begin{definition}
A representation $\varphi\colon A\to \mathrm{Hom}_\CC(V,V)$ is \emph{irreducible} if  $AV\ne\{0\}$ and there is no sub-$\CC$-vector space $W$ of $V$ with:
\[
\{0\}\ne W\:\:, \:\:W\ne V
\]
and 
\[
\omega w\in W\quad\forall(\omega, w)\in \k\times W
\]
and
\[
aw\in W\quad\forall(a, w)\in A \times W
\]
\end{definition}

\begin{definition}
Two representations of the $\k$-algebra $A$
\begin{align*}
\varphi_1\colon A\to \mathrm{Hom}_\CC(V_1,V_1)\\
\varphi_2\colon A\to \mathrm{Hom}_\CC(V_2,V_2)
\end{align*}
are \emph{equivalent} if there is an isomorphism of $\CC$-vector spaces\vspace{2mm}
\setlength{\abovedisplayskip}{0pt}
\setlength{\belowdisplayskip}{0pt}
\[
T\colon V_1\to V_2
\]
with
\[
T(av)=aT(v)\quad \forall\,(a,v)\in A\times V
\]
and 
\[
T(\omega v)=\omega T(v)\quad \forall\,(\omega,v)\in \k\times V\vspace{3mm}
\]

\noindent The \emph{spectrum} of $A$, also denoted Irr($A$), is the set of equivalence classes of irreducible representations of $A$.
\[
\mathrm{Irr}(A):=\{\mathrm{Irreducible\; representations\; of\; A}\}/\sim.
\]
\end{definition}

It can happen that the spectrum is empty.  For example, let $A$ comprise all $2$ by $2$ matrices of the form 
\[
 \left(
\begin{array}{cc}
0 & x\\
0 & 0
\end{array}
\right)
\]
with $x \in \C$.   Then $A$ is a commutative non-unital algebra of finite type (with $k = \C$) such that  $\Irr(A) = \emptyset$.

\section{On the action of $k$}

\noindent For a $k$-algebra $A$, $A_{\mathbb{C}}$ denotes the underlying $\mathbb{C}$-algebra of $A$. Then 
$A_{\mathbb{C}}$ is obtained from $A$ by forgetting the action of $k$ on $A$.
For $A_\mathbb{C}$ there are the usual definitions :  
A \emph{representation} of $A_{\mathbb{C}}$ [or a (left) \emph{$A_{\mathbb{C}}$-module}] is a $\CC$-vector space $V$ with a given morphism of $\CC$-algebras
\[
A_{\mathbb{C}}\longrightarrow \mathrm{Hom}_\CC(V,V)\vspace{3mm}
\]
An $A_{\mathbb{C}}$-module $V$ is \emph{irreducible} if  $A_{\mathbb{C}}V\ne\{0\}$ and there is no sub-$\CC$-vector space $W$ of $V$ with:
\[
\{0\}\ne W\:\:, \:\:W\ne V
\]
and 
\[
aw\in W\quad\forall(a, w)\in A_{\mathbb{C}} \times W\vspace{3mm}
\]
Two representations of  $A_{\mathbb{C}}$
\begin{align*}
\varphi_1\colon A\to \mathrm{Hom}_\CC(V_1,V_1)\\
\varphi_2\colon A\to \mathrm{Hom}_\CC(V_2,V_2)
\end{align*}
are \emph{equivalent} if there is an isomorphism of $\CC$-vector spaces\vspace{2mm}
\setlength{\abovedisplayskip}{0pt}
\setlength{\belowdisplayskip}{0pt}
\[
T\colon V_1\to V_2
\]
with
\[
T(av)=aT(v)\quad \forall\,(a,v)\in A\times V\vspace{3mm}
\]

\noindent Irr($A_{\mathbb{C}}$):=\{Irreducible representations of $A_{\mathbb{C}}$\}/$\sim$.\vspace{3mm}\\

\noindent An $A_{\mathbb{C}}$-module $V$ for which the following two properties are valid is \emph{strictly non-degenerate} 
\begin{itemize}
\item $A_{\mathbb{C}}V=V$
\item If $v\in V$ has $av=0\:\:\forall a\in A_{\mathbb{C}}$, then $v=0.$

\end{itemize}
\begin{lemma}
Any irreducible $A_{\mathbb{C}}$-module is strictly non-degenerate. 
\end{lemma} 

\begin{proof} Let $V$ be an irreducible $A_{\mathbb{C}}$-module.
First, consider $A_{\mathbb{C}}V\subset V$. $A_{\mathbb{C}}V$ is preserved by the action of $A_{\mathbb{C}}$
on $V$. Cannot have $A_{\mathbb{C}}V = \{0\}$ since this would contradict the irreducibility of $V$. Therefore  $A_{\mathbb{C}}V = V.$\\
Next, set 
\[
W = \{v\in V | av=0\:\:\forall a\in A_{\mathbb{C}}\}
\]
$W$ is preserved by the action of $A_{\mathbb{C}}$ on $V$. $W$ cannot be equal to $V$ since this would imply $A_{\mathbb{C}}V=\{0\}.$
Hence $W=\{0\}.$\hspace{53mm}
\end{proof}

\begin{lemma}\label{k-action}  Let $A$ be a $k$-algebra, and let $V$ be a strictly non-degenerate $A_{\mathbb{C}}$-module.
Then there is a unique unital morphism of $\mathbb{C}$ algebras
\[
k\to\mathrm{Hom}_{\mathbb{C}}(V, V)
\]
which makes $V$ an $A$-module.\\
\end{lemma}   

\begin{proof} Given $v\in V$, choose  $v_1,v_2,\ldots,v_r\in V$ and $a_1,a_2,\ldots,a_r\in A$ with 
\[
v=a_1v_1+a_2v_2+\cdots+a_rv_r
\]
For $\omega\in k$, define $\omega v$ by :
\[
\omega v=(\omega a_1)v_1+(\omega a_2)v_2+\cdots+(\omega a_r)v_r
\]
The second condition in the definition of strictly non-degenerate implies that $\omega v$ is well-defined.\hspace{90mm} 
\end{proof}

Lemma \ref{k-action}  will be referred to as the ``$k$-action for free lemma".

 If $V$ is an $A$-module, $V_{\mathbb{C}}$ will denote the underlying $A_{\mathbb{C}}$-module. $V_{\mathbb{C}}$ is 
obtained from $V$ by forgetting the action of $k$ on $V$.
\begin{lemma} If $V$ is any irreducible $A$-module, then $V_{\mathbb{C}}$ is an irreducible $A_{\mathbb{C}}$-module.

\end{lemma} 

\begin{proof}  Suppose that $V_{\mathbb{C}}$ is not an irreducible $A_{\mathbb{C}}$-module. Then there is a sub-$\mathbb{C}$-vector
space $W$ of $V$ with: 
\[ { 0 } \neq W , \quad W \neq V 
\]
and
\[
aw\in W\quad\;\;\forall (a, w)\in A\times W
\]
Consider $AW\subset W$. $AW$ is preserved by both the $A$-action on $V$ and the $k$-action on $V$. Thus if $AW\neq\{0\}$, then $V$ is not an irreducible
$A$-module. Hence $AW=\{0\}$. Consider $kW\supset W$. $kW$ is preserved by the $k$-action on $V$ and is also preserved by the $A$-action on $V$ because
$A$ annihilates $kW$. Since $A$ annihilates $kW$, cannot have $kW = V$. Therefore $\{0\}\neq kW$, $kW\neq V$, which contradicts the irreducibility of the $A$-module $V$. 
\end{proof}

\noindent A corollary of Lemma \ref{k-action} is :
\begin{corollary}
For any $k$-algebra $A$,  the map\vspace{2mm} 
\[
\mathrm{Irr}(A)\rightarrow \mathrm{Irr}(A_{\mathbb{C}})
\]
\[
V\mapsto V_{\mathbb{C}}
\]
is a bijection.
\end{corollary}

\begin{proof} Surjectivity follows from  Lemmas 3.1 and 3.2.  For injectivity, let $V , W$ be two irreducible $A$-modules such that $V_{\mathbb{C}}$ and
$W_{\mathbb{C}}$ are equivalent  $A_{\mathbb{C}}$-modules. Let $T\colon V\rightarrow W$ be an isomorphism of $\mathbb{C}$ vector spaces with
\[
T(av)=aT(v)\quad \forall\,(a,v)\in A\times V
\] 
Given $v\in V$ and $\omega\in k$, choose $v_1,v_2,\ldots,v_r\in V$ and $a_1,a_2,\ldots,a_r\in A$ with 
\[
v=a_1v_1+a_2v_2+\cdots+a_rv_r
\]
Then
\begin{align*}
T(\omega v) & = T((\omega a_1)v_1+(\omega a_2)v_2+\cdots +(\omega a_r)v_r) \\
& = (\omega a_1)Tv_1+(\omega a_2)Tv_2+\cdots + (\omega a_r)Tv_r \\
& = \omega(a_1Tv_1+a_2Tv_2+\cdots + a_rTv_r)\\
 & = \omega (Tv).
\end{align*}

Hence $T\colon V\rightarrow W$ intertwines the $k$-actions on $V, W$ and thus $V, W$ are equivalent $A$-modules.
\end{proof}

\section{Finite type $k$-algebras}

\noindent An ideal $I$ in a $k$-algebra $A$ is \emph{primitive} if $I$ is the null-space of an irreducible representation of $A$, i.e. there is
an  irreducible representation  of $A$ \linebreak 
$\varphi\colon A\to \mathrm{Hom}_\CC(V,V)$ with
\[
I = \{a\in A\; |\; \varphi(a) = 0 \}
\]
Prim($A$) denotes the set of all primitive ideals in $A$. The evident map
\[
\mathrm{Irr}(A)\rightarrow \mathrm{Prim}(A)
\]
sends an irreducible representation to its null-space. On Prim($A$) there is the Jacobson topology. 
If $S$ is any subset of $\mathrm{Prim}(A)$, $S\subset \mathrm{Prim}(A)$, then the closure $\overline{S}$ of $S$ is :\\
\[
\overline{S}:=\left\{I\in \mathrm{Prim}(A)\mid I\supset \cap_{L\in S} L\right\} \vspace{3mm}
\]

\begin{definition}A $k$-algebra $A$ is of \emph{finite type} if, as a $k$-module, $A$ is finitely generated. 
\end{definition}

For any finite type $k$-algebra $A$, the following three statements are valid :
\begin{itemize}
\item  If $\varphi\colon A\to \mathrm{Hom}_\CC(V,V)$ is any irreducible representation of $A$, then $V$ is a finite dimensional $\mathbb{C}$ vector space and $\varphi\colon A\to \mathrm{Hom}_\CC(V,V)$  is surjective.
\item The evident map  $\mathrm{Irr}(A)\rightarrow \mathrm{Prim}(A)$ is a bijection.
\item Any primitive ideal in $A$ is a maximal ideal.

\end{itemize}
Since $\mathrm{Irr}(A)\rightarrow \mathrm{Prim}(A)$ is a bijection, the Jacobson topology on Prim($A$) can be transferred to Irr($A$) and thus Irr($A$) is topologized. Equivalently,
Irr($A$) is topologized by requiring that $\mathrm{Irr}(A)\rightarrow \mathrm{Prim}(A)$ be a homeomorphism.



For a finite type $k$-algebra $A$ ($k=\mathcal{O}(X))$, the central character is a  map
\[
\mathrm{Irr}(A)\longrightarrow X
\]
defined as follows. Let $\varphi$
\[
A\longrightarrow \mathrm{Hom}_\CC(V,V)
\]
\[
\k\longrightarrow \mathrm{Hom}_\CC(V,V)
 \] 
be an irreducible representation of $A$.
 $I_V$ denotes the identity operator of $V$.

For $\omega\in \k=\mathcal{O}(X)$, define $T_\omega\colon V\to V$
by
\[
T_\omega(v)=\omega v
\]
for all $v \in V$.
Then $T_{\omega}$ is an intertwining operator for  $A\rightarrow \mathrm{Hom}_\CC(V,V)$.  By Lemma 3.3 and Schur's Lemma, $T_{\omega}$ is a scalar multiple of $I_V$.
\[
T_{\omega} = \lambda_{\omega}I_V\hspace{8mm} \lambda_{\omega}\in\CC
\]
The map
\[
\omega\mapsto\lambda_{\omega}
\]
is a unital morphism of $\mathbb{C}$ algebras $\mathcal{O}(X)\rightarrow\mathbb{C}$ and thus is given by evaluation at a unique ($\mathbb{C}$ rational) point $p_{\varphi}$  of $X$:
\[
\lambda_{\omega} = \omega(p_{\varphi})\hspace{8mm}\forall\omega\in\mathcal{O}(X)
\]
The \emph{central character} $\mathrm{Irr}(A)\longrightarrow X$ is
\[
\varphi\mapsto p_{\varphi}\vspace{5mm}
\]
\textsc{Remark.} Corollary 3.4 states that Irr($A$) depends only on the underlying $\mathbb{C}$ algebra $A_{\mathbb{C}}$. The central character Irr($A$)$\rightarrow X$, however,
does depend on the structure of $A$ as a $k$-module. A change in the action of $k$ on $A_{\mathbb{C}}$ will change the central character.\vspace{3mm}\\
The central character $\mathrm{Irr}(A)\rightarrow X$ is continuous where Irr($A$) is topologized as above and $X$ has the Zariski topology. For a proof of this assertion
see \cite[Lemma 1, p.326]{KNS}.
From a somewhat heuristic non-commutative geometry point of view, $A_{\mathbb{C}}$ is a non-commutative complex affine variety, and a given action of $k$ on $A_{\mathbb{C}}$, making $A_{\mathbb{C}}$
into a finite type $k$-algebra $A$, determines a morphism of algebraic varieties $A_{\mathbb{C}} \rightarrow X$.
\section{Morita equivalence for $k$-algebras}
\begin{definition}
Let $B$ be a $\k$-algebra. A \emph{right $B$-module} is a $\CC$-vector space $V$ with given morphisms of $\CC$-algebras
\[
B^{\op}\longrightarrow\mathrm{Hom}_\CC(V,V)
\]
\[
\k\longrightarrow\mathrm{Hom}_\CC(V,V)
\]
such that:
\begin{enumerate}
\item $\k\rightarrow\mathrm{Hom}_\CC(V,V)$ is unital
\end{enumerate} 
\begin{enumerate}
\setcounter{enumi}{1}
\item $v(\omega b)=(v\omega)b=(vb)\omega$\quad $\forall (v,\omega,b)\in V\times\k\times B$.
\end{enumerate}
$B^{\op}$ is the opposite algebra of $B$. $V$ is non-degenerate  if $VB=V$.
\end{definition}

Note that a right $B$-module is the same as a left $B^{\op}$-module.

 With $k$ fixed, let $A$, $B$ be two $\k$-algebras. An $A-B$ bimodule, denoted ${}_A V_B$, is a $\mathbb{C}$ vector space $V$ such that :
\begin{enumerate}
\item $V$ is a left $A$-module.
\item $V$ is a right $B$-module.
\item $a(vb)=(av)b\qquad \forall \, (a,v,b)\in A\times V\times B$.
\item $\omega v=v\omega\qquad \forall \, (\omega,v)\in \k\times V$.\vspace{3mm}
\end{enumerate}
An $A-B$ bimodule ${}_AV_B$ is non-degenerate if $AV=V=VB$. $I_V$ is the identity map of $V$. $A$ is an $A-A$ bimodule in the evident way.
\noindent 
\begin{definition}A $k$-algebra 
$A$ has \emph{local units} if given any finite set $\{a_1,a_2,\ldots,a_r\}$ of elements of $A$, there is an idempotent $Q\in A$ ($Q^2=Q$) with
\[
Qa_j=a_jQ=a_j\qquad j=1,2,\ldots,r.\vspace{3mm}
\]
\end{definition}

\begin{definition}Let $A$, $B$ be  two $\k$-algebras with local units.
A \emph{Morita equivalence} (between $A$ and $B$) is given by a pair of non-degenerate bimodules
\[
{}_AV_B\quad {}_BW_A
\]
together with isomorphisms of bimodules
\begin{align*}
\alpha\colon V\otimes_B W&\to A\\
\beta\colon W\otimes _A V &\to B
\end{align*}
such that there is commutativity in the diagrams:
\end{definition}
\[
\xymatrix{
V\otimes_B W\otimes _A V \ar[rr]^<<<<<<<<<<{I_V\otimes \beta} \ar[d]^{\alpha\otimes I_V}&& V \otimes_B B \ar[d]^{\cong}\\
A\otimes_A V\ar[rr]_{\cong} && V
}
\]

\[
\xymatrix{
W\otimes_A V\otimes _B W \ar[rr]^<<<<<<<<<<{I_W\otimes \alpha} \ar[d]^{\beta\otimes I_W}&& W \otimes_A A \ar[d]^{\cong}\\
B\otimes_B V\ar[rr]_{\cong} && W
}\vspace{7mm}
\]

  Let $A$, $B$ two $\k$-algebras with local units, and suppose given a Morita equivalence
\[
{}_AV_B\hspace{7mm} {}_BW_A\hspace{7mm}\alpha\colon V\otimes_B W\rightarrow A\hspace{7mm}\beta\colon W\otimes _A V\rightarrow B\vspace{2mm}   
\]  
  
The \emph{linking algebra} is 

\[
L({}_AV_B,{}_BW_A): = \left(
\begin{array}{cc}
A & V\\
W & B
\end{array}
\right)
\]

i.e.\;$L({}_AV_B,{}_BW_A)$ consists of all $2\times 2$ matrices having (1, 1) entry in $A$,\\ (2, 2) entry in $B$, (2, 1) entry in $W$, and (1, 2) entry in $V$.
Addition and multiplication are matrix addition and matrix multiplication.  Note that $\alpha$ and $\beta$ are used in the matrix multiplication.\vspace{2mm}\\
$L({}_AV_B,{}_BW_A)$ is a $k$-algebra. With $\omega \in k$, the action of $k$ on $L({}_AV_B,{}_BW_A)$ is given by

\[
   \omega \left(
\begin{array}{cc}
a & v\\
w & b
\end{array}
\right) =  \left(
\begin{array}{cc}
\omega a & \omega v\\
\omega w & \omega b
\end{array}
\right)
\]

\bigskip 

A Morita equivalence between $A$ and $B$ 
determines an equivalence of categories between the category of non-degenerate left $A$-modules  and the category of non-degenerate left $B$-modules. Similarly
for right modules. Also, a Morita equivalence determines  isomorphisms (between $A$ and $B$) of Hochschild homology, cyclic homology, and periodic cyclic homology.

\textsc{Example.} For $n$ a positive integer, let $\Mat_n(A)$ be  the $k$-algebra of all $n\times n$ matrices with entries in $A$.  If $A$ has local units, $A$ and $\Mat_n(A)$ are Morita equivalent as follows.  For $m, n$ positive integers, denote by $\Mat_{m, n}(A)$
the set of all $m\times n$  (i.e. $m$ rows and $n$ columns) matrices with entries in $A$. Matrix multiplication then  gives a map 
\[
\Mat_{m, n}(A)\times \Mat_{n, r}(A)\longrightarrow \Mat_{m, r}(A) \vspace{2mm}
\]
With this notation, $\Mat_{n, n}(A) = \Mat_{n}(A)$ and $\Mat_{1, 1}(A) = \Mat_{1}(A) = A.$ Hence matrix multiplication gives maps\vspace{.1mm}\\  
\[
\Mat_{1,n}(A)\times \Mat_n(A)\longrightarrow \Mat_{1, n}(A)\hspace{10mm} \Mat_n(A)\times \Mat_{n, 1}\longrightarrow \Mat_{n, 1}(A)\vspace{2mm}
\]
Thus $\Mat_{1, n}(A)$ is a right  $\Mat_n(A)$-module and $\Mat_{n, 1}(A)$ is a left $\Mat_n(A)$-module.\\
Similarly, $\Mat_{1, n}(A)$ is a left  $A$-module and $\Mat_{n, 1}(A)$ is a right $A$-module.\\
With $A=A$ and $B=\Mat_n(A)$, the bimodules of the Morita equivalence  are $V= \Mat_{1, n}(A)$ and $W = \Mat_{n, 1}(A)$.\\
Note that the required isomorphisms of bimodules
\begin{align*}
\alpha\colon V\otimes_B W&\to A\\
\beta\colon W\otimes _A V &\to B
\end{align*}
are obtained by observing that the matrix multiplication maps\vspace{1mm}
\begin{align*}
\Mat_{1, n}(A) \times \Mat_{n, 1}(A)&\to A\\
\Mat_{n, 1}(A)\times \Mat_{1, n}(A) &\to \Mat_n(A)
\end{align*}
factor through the quotients  $\Mat_{1, n}(A) \otimes_{\Mat_n(A)} \Mat_{n, 1}(A),\quad \Mat_{n, 1}(A) \otimes_A \Mat_{1, n}(A)$ \\
and so give bimodule isomorphisms
\begin{align*}
\alpha\colon \Mat_{1, n}(A) \otimes_{\Mat_n(A)} \Mat_{n, 1}(A)&\to A\\
\beta\colon \Mat_{n, 1}(A) \otimes_A \Mat_{1, n}(A)&\to \Mat_n(A)
\end{align*}

\noindent If $A$ has local units, then $\alpha$ and $\beta$ are isomorphisms. Therefore $A$ and $\Mat_n(A)$ are Morita equivalent.\vspace{2mm}\\
If $A$ does not have local units, then $\alpha$ and  $\beta$ can fail to be isomorphisms, and there is no way to prove that $A$ and $\Mat_n(A)$ 
 are Morita equivalent. In examples, this already happens with $n=1$, and there is then no way to prove (when $A$ does not have local units) that
 $A$ is Morita equivalent to $A$.  For more details on this issue see  below where the proof is given that in the new equivalence relation $A$
 and  $\Mat_n(A)$  are equivalent even when $A$ does not have local units.\vspace{2mm}\\
 A finite type $k$-algebra $A$ has local units iff $A$ is unital.
 
 \section{The exceptional set $\fE_A(X)$}   A Morita equivalence between two finite type $k$-algebras $A, B$ preserves the central character i.e. there is commutativity in the diagram
\[
\begin{CD}
\mathrm{Irr}(A) @>>> \mathrm{Irr}(B) \\
@V cc VV         @VV cc V \\
X @>>I_X> X
\end{CD}
\]
where the upper horizontal arrow is the bijection determined by the given Morita equivalence, the two vertical arrows are the two central characters, and $I_X$ is the identity map of $X$.

By the \emph{exceptional set} $\fE(A)$ we shall mean the set of all $x \in X$ such that the fibre over $x$ has cardinality greater than $1$:
\[
\fE(A): = \{x \in X : \sharp cc^{-1}(x) >1\}.
\]   

If $A$ and $B$ are Morita equivalent as $k$-algebras, then we will have $\fE(A) = \fE(B)$.  Contrapositively, we have the following useful
\begin{corollary}\label{exceptional}  If $\fE (A) \neq \fE(B)$ then A and B are  not  Morita  equivalent as $k$-algebras.
\end{corollary}

 \section{Spectrum preserving morphisms}
 
\noindent Let $A$, $B$ two finite type $\k$-algebras,  and let $f\colon A\to B$ be a morphism of $\k$-algebras.
\begin{definition}
$f$ is \emph{spectrum preserving} if
\begin{enumerate}
\item Given any primitive ideal $J\subset B$, there is a unique primitive ideal $I\subset A$ with $I\supset f^{-1}(J)$
\end{enumerate} 
and
\begin{enumerate}
\setcounter{enumi}{1}
\item The resulting map 
\[
\mathrm{Prim}(B)\to \mathrm{Prim}(A)
\]
is a bijection.
\end{enumerate}
\end{definition}

For $I$ and $J$ as above, let $V$ be the simple $B$-module with annihilator $J$.   Then the simple $A$-module
with annihilator $I$ can be realized as a quotient of $V$ and also as a subspace of $V$.   

\begin{example} Let $A$, $B$ two unital finite type $\k$-algebras, and suppose given a Morita equivalence
\[
{}_AV_B\hspace{7mm} {}_BW_A\hspace{7mm}\alpha\colon V\otimes_B W\rightarrow A\hspace{7mm}\beta\colon W\otimes _A V\rightarrow B\vspace{2mm}   
\]    
With the linking algebra $L({}_AV_B,{}_BW_A)$ as above, the inclusions

\[
A\hookrightarrow L({}_AV_B,{}_VW_A)\hookleftarrow B
\]

\[
a \mapsto \left(
\begin{array}{cc}
a & 0\\
0 & 0
\end{array}
\right)
\quad \quad
 \left(
\begin{array}{cc}
0 & 0\\
0 & b
\end{array}
\right) \leftmapsto b
\]
are spectrum preserving morphisms of finite type $\k$-algebras.
The bijection
\[
\mathrm{Prim}(B)\longleftrightarrow \mathrm{Prim}(A)
\]
so obtained is the bijection determined by the given Morita equivalence.
\end{example}

\noindent\underline{Remark.} If  $f\colon A\to B$ is a spectrum preserving morphism of finite type $\k$-algebras, then 
the resulting bijection
\[
\mathrm{Prim}(B)\longleftrightarrow \mathrm{Prim}(A)
\]
is a homeomorphism. For a proof of this assertion see \cite[Theorem 3, p.342]{BN}. 
Consequently, if 
$A$, $B$ are two unital finite type $\k$-algebras, and 
\[
{}_AV_B\hspace{7mm} {}_BW_A\hspace{7mm}\alpha\colon V\otimes_B W\rightarrow A\hspace{7mm}\beta\colon W\otimes _A V\rightarrow B\vspace{2mm}   
\]    
is a Morita equivalence, then the resulting bijection
\[
\mathrm{Prim}(B)\longleftrightarrow \mathrm{Prim}(A)
\]
is a homeomorphism.\\

\begin{definition}
An ideal $I$ in a $k$-algebra $A$ is a \emph{$\k$-ideal} if $\omega a\in I$ $\forall\,(\omega,a)\in \k\times I$.
\end{definition}
\noindent\underline{Remark.} Any primitive ideal in a $k$-algebra $A$ is a $k$-ideal.\\

\noindent Given $A, B$ two finite type $\k$-algebras, let $f\colon A\to B$ be a morphism of $\k$-algebras.
\begin{definition}
$f$ is \emph{spectrum preserving with respect to filtrations} if there are $\k$-ideals  
\[
0=I_0\subset I_1\subset\cdots\subset I_{r-1}\subset I_r=A\qquad \text{in $A$} 
\]
and $\k$ ideals
\[
0=J_0\subset J_1\subset\cdots\subset J_{r-1}\subset J_r=B\qquad \text{in $B$}
\]
with $f(I_j)\subset J_j$, ($j=1,2,\ldots,r$) and $I_j/I_{j-1}\to J_j/J_{j-1}$, ($j=1,2,\ldots,r$) is spectrum preserving.
\end{definition}

The primitive ideal spaces of the subquotients $I_j/I_{j-1}$ and $ J_j/J_{j-1}$ are the strata for stratifications of $\Prim(A)$ and $\Prim(B)$.
Each stratum of $\Prim(A)$ is mapped homeomorphically onto the corresponding stratum of $\Prim(B)$.   However, the map 
\[
\Prim(A) \to \Prim(B)
\]
might not be a homeomorphism.

\section{Algebraic Variation of $\k$-structure}

\noindent Let $A$ be a unital $\mathbb{C}$-algebra, and let 
\[
\Psi\colon \k\to Z\left(A[t,t^{-1}]\right)
\]
be a unital morphism of $\mathbb{C}$-algebras.  Here $t$ is an indeterminate,so $A[t, t^{-1}]$ is the algebra of Laurent polynomials with coefficients in $A$. As above Z denotes  ``center". For $\zeta\in \mathbb{C}^{\times}=\mathbb{C}-\{0\}$, $ev(\zeta)$ denotes the ``evaluation at $\zeta$'' map:
\begin{align*}
ev(\zeta)\colon A[t,t^{-1}]&\to A\\
\sum a_jt^j                         &\mapsto \sum a_j\zeta^j 
\end{align*}

\noindent Consider the composition
\[
\k\overset{\Psi}{\longrightarrow} Z\left(A[t,t^{-1}]\right)\overset{ev(\zeta)}{\longrightarrow} Z(A).
\]
Denote the unital $\k$-algebra so obtained by $A_\zeta$. $\forall\zeta\in\CC^\times=\CC-\{0\}$, the underlying $\CC$-algebra of $A_\zeta$ is $A$.\vspace{2mm}\\
\[
(A_{\zeta})_{\mathbb{C}} = A\hspace{20mm}\forall\zeta\in\mathbb{C}^{\times}\vspace{2mm}
\]
Such a family $\{A_{\zeta}\}$, $\zeta\in\mathbb{C}^{\times}$, of unital $k$-algebras, will be referred to as \emph{an algebraic variation of k-structure with parameter
space }$\mathbb{C}^{\times}.$ 

\section{Stratified Equivalence}
\noindent With $k$ fixed, consider the collection of all finite type $k$-algebras. On this collection, \emph{stratified equivalence} is, by definition, the equivalence relation generated
by the two elementary steps :\vspace{"3mm}\\
\underline{Elementary Step 1.} If there is a morphism of $k$-algebras $f: A\rightarrow B$ which is spectrum preserving with respect to filtrations, then $A\sim B.$\vspace{3mm}\\
\underline{Elementary Step 2.} If there is $\{A_{\zeta}\}$, $\zeta\in\mathbb{C}^{\times}$,  an algebraic variation of $k$-structure with parameter
space $\mathbb{C}^{\times}$,  such that each $A_{\zeta}$  is a unital finite type $k$-algebra, then for any $\zeta, \eta\in\mathbb{C}^{\times}, A_{\zeta}\sim A_{\eta}.$\vspace{3mm}\\
Thus, two finite type $k$-algebras $A, B$ are equivalent if and only if there is a finite sequence $A_0, A_1, A_2,\ldots,A_r$ of finite type $k$-algebras  with $A_0 =A, A_r=B$, and for each $j = 0, 1, \ldots, r-1$ one of the following three possibilities is valid :\\
\begin{itemize}
\item a morphism of $k$-algebras $A_j\rightarrow A_{j+1}$ is given which is spectrum preserving with respect to filtrations.\\
\item a morphism of $k$-algebras $A_j\leftarrow A_{j+1}$ is given which is spectrum preserving with respect to filtrations.\\
\item $\{A_{\zeta}\}$, $\zeta\in\mathbb{C}^{\times}$,  an algebraic variation of $k$-structure with parameter
space $\mathbb{C}^{\times}$, is given such that each $A_{\zeta}$  is a unital finite type $k$-algebra, and $\eta, \tau\in\mathbb{C}^{\times}$
have been chosen with $A_j=A_{\eta}, A_{j+1}=A_{\tau}$.\vspace{3mm}
\end{itemize}
To give a stratified  equivalence relating $A$ and $B$, the finite sequence of elementary 
steps (including the filtrations)
must be given. Once this has been done, a bijection of the primitive ideal spaces and 
an isomorphism of periodic cyclic homology \cite{BN} are determined:
\[
\mathrm{Prim}(A)\longleftrightarrow \mathrm{Prim}(B)\quad\quad\quad \HP_*(A)\cong \HP_*(B)
\]
\begin{proposition}\label{Morita}
If two unital finite type $k$-algebras $A, B$ are Morita equivalent 
(as $k$-algebras) then they are stratified  equivalent.
\[
 A\underset{\text{\tiny Morita}}{\sim}B\Longrightarrow A\sim B
\]
\end{proposition}

\begin{proof} Let $A$, $B$ two unital finite type $\k$-algebras, and suppose given a Morita equivalence
\[
{}_AV_B\hspace{7mm} {}_BW_A\hspace{7mm}\alpha\colon V\otimes_B W\rightarrow A\hspace{7mm}\beta\colon W\otimes _A V\rightarrow B\vspace{2mm}   
\]    
The linking algebra is
\[
L({}_AV_B,{}_BW_A): = \left (
\begin{array}{cc}
A & V\\
W & B
\end{array}
\right )
\]

\noindent The inclusions
\[
A\hookrightarrow L({}_AV_B,{}_VW_A)\hookleftarrow B
\]

\[
a \mapsto \left(
\begin{array}{cc}
a & 0\\
0 & 0
\end{array}
\right)
\quad \quad
 \left(
\begin{array}{cc}
0 & 0\\
0 & b
\end{array}
\right) \leftmapsto b
\]
are spectrum preserving morphisms of finite type $\k$-algebras. 
Hence $A$ and $B$ are stratified  equivalent.
\end{proof}

\noindent According to the above, a Morita equivalence of $A$ and $B$ gives a homeomorphism
\[
 \mathrm{Prim}(A) \simeq \mathrm{Prim}(B) 
 \]
 However, the bijection 
\[
\mathrm{Prim}(A)\longleftrightarrow \mathrm{Prim}(B)
\]
obtained from a stratified equivalence might not be a homeomorphism, as in the following example.

\begin{example}\label{XY}
   Let $Y$ be   a sub-variety of $X$.    
We will write $\mathcal{I}_Y$ for the ideal in $\mathcal{O}(X)$ determined by $Y$, so that 
\[
\mathcal{I}_Y=\left\{\omega\in\mathcal{O}(X)\mid \omega(p)=0\,\,\,\,\,\forall\,p\in Y\right\}.
\]
Let $A$ be the algebra of all $2\times 2$ matrices whose diagonal entries  are in $\mathcal{O}(X)$ and whose
off-diagonal entries are in $\mathcal{I}_Y$. Addition and multiplication in $A$ are matrix addition and matrix
multiplication. As a $k$-module, $A$ is the direct sum of $\mathcal{O}(X)\oplus\mathcal{O}(X)$ with $\mathcal{I}_Y\oplus\mathcal{I}_Y$.

\[
A = \left (
\begin{array}{cc}
\cO(X) & \cI_Y\\
\cI_Y & \cO(X)
\end{array}
\right )
\]

 Set $B=\mathcal{O}(X)\oplus \mathcal{O}(Y)$, so that  $B$ is the coordinate algebra of the disjoint union $X\sqcup Y$. We have
$\mathcal{O}(Y)=\mathcal{O}(X)/\mathcal{I}_Y.$
As a $k=\mathcal{O}(X)$-module, $B$ is the direct sum $\mathcal{O}(X)\oplus (\mathcal{O}(X)/\mathcal{I}_Y)$.
\end{example}

\begin{theorem} The algebras $A$ and $B$  are stratified  equivalent but not Morita equivalent: 
\[
A\sim B\qquad A\underset{\text{\tiny Morita}}{\nsim}B
\]
\end{theorem}

\begin{proof} Let $\Mat_2(\mathcal{O}(X)$ denote the algebra of all $2\times 2$ matrices with entries in $\mathcal{O}(X)$. Consider
the algebra morphisms 

\[
A \longrightarrow \Mat_2(\cO(X) \oplus \cO(Y)) \longleftarrow \cO(X) \oplus \cO(Y)
\]

\[
T \mapsto (T, t_{22} | Y) \quad \quad \quad (T_{\omega}, \theta) \leftmapsto (\omega, \theta)
\]
where
\[
T = \left( 
\begin{array}{cc}
t_{11} & t_{12}\\
t_{21} & t_{22}
\end{array}
\right ) \qquad T_{\omega} = \left(\begin{array}{cc}
\omega & 0\\
0 & 0
\end{array}
\right )
\]

The filtration of $A$ is given by
\[
\{0\} \subset \left( \begin{array}{cc}
\cO(X) & \cI_Y\\
\cI_Y & \cI_Y
\end{array}
\right )
\subset A
\]
and the filtration of $\Mat_2(\cO(X)) \oplus \cO(Y)$ is given by

\[
\{0\} \subset \Mat_2(\cO(X) \oplus \{0\}) \subset \Mat_2(\cO(X)) \oplus \cO(Y).
\]

The rightward pointing arrow is spectrum preserving with respect to the indicated filtrations. The leftward pointing arrow is spectrum preserving (no filtrations needed).
We infer that
\[
A\sim B.
\]

Note that 
\begin{align*}
\Prim(A) & =
\text{\scalebox{.9}{$X$ with each point of $Y$ replaced by two points}}
\end{align*}
and 
\begin{align*}
\Prim(B) & =\Prim\big(\OOO(X)\oplus\OOO(Y)\big)\\
& =X\sqcup Y
\end{align*}

The spaces $\Prim(A)$ and $\Prim(B)$ are not homeomorphic, and therefore $A$ is not Morita equivalent to $B$.
\end{proof}

Unlike Morita equivalence, stratified equivalence works well for finite
 type $k$-algebras  whether or not the algebras are unital, e.g. $A$ and $\Mat_n(A)$ are stratified  equivalent even when $A$ is not unital. 
 See Proposition 8.5
 below.
 
 For any $k$-algebra $A$  there is the evident isomorphism of $k$-algebras $\Mat_n(A)\cong A\otimes_{\mathbb{C}}\Mat_n(\mathbb{C})$. Hence, using this
 isomorphism, if $W$ is a representation of  $A$ and $U$ is a representation of $\Mat_n(\mathbb{C})$, then $W\otimes_{\mathbb{C}}U$ is a representation of $\Mat_n(A)$.
 
  As  above, $\Mat_{n, 1}(\mathbb{C})$ denotes the $n\times 1$ 
 matrices with entries in $\mathbb{C}$. Matrix multiplication gives the usual action of $\Mat_n(\mathbb{C})$ on $\Mat_{n, 1}(\mathbb{C})$.
 \[
 \Mat_n(\mathbb{C})\times \Mat_{n, 1}(\mathbb{C})\longrightarrow \Mat_{n, 1}(\mathbb{C})
 \]
 This is the unique irreducible representation of $\Mat_n(\mathbb{C}).$ For any $k$-algebra $A$, if $W$ is a representation of $A$, then $W\otimes_{\mathbb{C}}\Mat_{n, 1}(\mathbb{C})$
 is a representation of $\Mat_n(A)$.
 \begin{lemma}  Let $A$ be a finite type $k$-algebra and let $n$ be a positive integer. Then:\\
 (i) If $W$ is an irreducible representation of $A$, $W\otimes_{\mathbb{C}} \Mat_{n, 1}(\mathbb{C})$ is an irreducible representation of $\Mat_n(A).$\\
 (ii) The resulting map $\Irr(A)\rightarrow \Irr(\Mat_n(A))$ is a bijection.   
 \end{lemma}
 \noindent\underline{Proof.} For (i), suppose given an irreducible representation $W$ of $A$. Let $J$ be the primitive ideal in $A$ which is the null space of $W$. Then the null space of   $W\otimes_{\mathbb{C}} \Mat_{n, 1}(\mathbb{C})$  
 is $J\otimes_{\mathbb{C}}\Mat_n(\mathbb{C}).$Consider  the quotient algebra\\
  $A\otimes_{\mathbb{C}}\Mat_n(\mathbb{C})/J\otimes_{\mathbb{C}}\Mat_n(\mathbb{C}) = (A/J)\otimes_{\mathbb{C}}\Mat_n(\mathbb{C})$.
  This is isomorphic to $\Mat_{rn}(\mathbb{C})$ where $A/J\cong \Mat_r(\mathbb{C})$, and so $W\otimes_{\mathbb{C}}\Mat_{n, 1}(\mathbb{C})$ is irreducible.

 \begin{proposition} Let $A$ be a finite type $k$-algebra and let $n$ be a positive integer, then $A$ and $\Mat_n(A)$ are stratified equivalent. 
 \end{proposition}

\begin{proof} Let   $f\colon A\rightarrow \Mat_n(A)$ be the morphism of $k$-algebras which maps $a\in A$ to the diagonal matrix
\begin{equation*}
   \left[
   \begin{matrix}
   a  &  0  & \dots  &  0\\
   0   &  a  & \dots  &  0\\
   \vdots  &  \vdots &  \ddots  &  \vdots\\
   0  &  0  &  \dots  & a
   \end{matrix}
   \right]
   \end{equation*}
 It will suffice to prove that  $f\colon A\rightarrow \Mat_n(A)$ is spectrum preserving. \\
 Let $J$ be an ideal in $A$. Denote by $J^{\Diamond}$ the ideal in 
 $\Mat_n(A)$ consisting of all $[a_{i j}]\in \Mat_n(A)$ such that each $a_{ij}$ is in $J$. Equivalently, $\Mat_n(A)$ is $A\otimes_{\mathbb{C}}\Mat_n(\mathbb{C})$
 and $J^{\Diamond}=J\otimes_{\mathbb{C}}\Mat_n(\mathbb{C})$. It will suffice to prove 
 \begin{enumerate}
 \item If $J$ is a primitive ideal in $A$, then $J^{\Diamond}$ is a primitive ideal in $\Mat_n(A)$.
 \item If $L$ is any primitive ideal in $\Mat_n(A)$, then there is a primitive ideal $J$ in $A$ with  $L=J^{\Diamond}$.  
 \end{enumerate}
 For (1), $J$ primitive $\Longrightarrow  J^{\Diamond}$ primitive,  because the quotient algebra
  $\Mat_n(A)/J^{\Diamond}$  is $(A/J)\otimes_{\mathbb{C}}\Mat_n(\mathbb{C})$ which is (isomorphic to) $\Mat_{rn}(\mathbb{C})$ where $A/J\cong \Mat_r(\mathbb{C})$.
\vspace{2mm}\\
For (2), since $\mathbb{C}$ is commutative, the action of $\mathbb{C}$ on $A$ can be viewed as both a left and right action. Matrix multiplication then gives 
a left and a right action of $\Mat_n(\mathbb{C})$ on $\Mat_n(A)$\vspace{2mm}\\
\[
\Mat_n(\mathbb{C})\times \Mat_n(A)\rightarrow \Mat_n(A)
\]
\[
\Mat_n(A)\times \Mat_n(\mathbb{C})\rightarrow \Mat_n(A)
\]

\noindent for which the associativity rule
\[
(\alpha\theta)\beta = \alpha(\theta\beta)\hspace{10mm}\alpha, \beta\in \Mat_n(A)\hspace{5mm}\theta\in \Mat_n(\mathbb{C})
\]
is valid.\\
If $V$ is any representation of $\Mat_n(A)$, the associativity rule 
\[
(\alpha\theta)(\beta v) = \alpha[(\theta\beta)v]\hspace{10mm}\alpha, \beta\in \Mat_n(A)\hspace{5mm}\theta\in \Mat_n(\mathbb{C})\hspace{5mm} v\in V
\]
is valid.\vspace{2mm}\\
Now let $V$ be an irreducible representation of $\Mat_n(A)$, with $L$ as its null-space. Define a (left) action
\[
\Mat_n(\mathbb{C})\times V\rightarrow V
\]
of $\Mat_n(\mathbb{C})$ on $V$ by proceeding as in the proof of Lemma 4.2  (the ``$k$-action for free" lemma) i.e. given $v\in V$, choose  $v_1,v_2,\ldots,v_r\in V$ and $\alpha_1,\alpha_2,\ldots,\alpha_r\in \Mat_n(A)$ 
with 
\[
v=\alpha_1v_1+\alpha_2v_2+\cdots+\alpha_rv_r  
\]
For $\theta\in \Mat_n(\mathbb{C})$, define $\theta v$ by :
\[
\theta v=(\theta \alpha_1)v_1+(\theta \alpha_2)v_2+\cdots+(\theta \alpha_r)v_r
\]
The strict non-degeneracy, Lemmas 4.1 and 4.3, of $V$ implies that $\theta v$ is well-defined as follows. Suppose that  $u_1,u_2,\ldots,u_s\in V$ and $\beta_1,\beta_2,\ldots,\beta_r\in \Mat_n(A)$ 
are chosen with 
\[
v=\alpha_1v_1+\alpha_2v_2+\cdots+\alpha_rv_r  = \beta_1u_1+\beta_2u_2+\cdots+\beta_su_s 
\]
If $\alpha$ is any element of $\Mat_n(A)$, then\\
$\alpha [(\theta \alpha_1)v_1+(\theta \alpha_2)v_2+\cdots+(\theta \alpha_r)v_r-(\theta \beta_1)u_1-(\theta\beta_2)u_2-\cdots-(\theta\beta_s)u_s] =$
$(\alpha\theta) [\alpha_1v_1+\alpha_2v_2+\cdots+\alpha_rv_r-\beta_1u_1-\beta_2u_2-\cdots-\beta_su_s] = (\alpha\theta)[v-v] = 0$\vspace{2mm}\\
Use $f\colon A\rightarrow \Mat_n(A)$ to make $V$ into an $A$-module
\[
av:= f(a)v\hspace{16mm}a\in A\hspace{6mm} v\in V
\]
The actions of $A$ and $\Mat_n(\mathbb{C})$ on $V$ commute. Thus for each $\theta\in \Mat_n(\mathbb{C}), \theta V$ is a sub-A-module of $V$, where 
$\theta V$ is the image of $v\mapsto\theta v$. Denote by $E_{ij}$ the matrix in $\Mat_n(\mathbb{C})$ which has $1$ for its $(i,j)$ entry and zero for all its other
entries. Then, as an $A$-module, $V$ is the direct sum
\[
V = E_{11}V\oplus E_{22}V\oplus\cdots\oplus E_{nn}V
\]
Moreover, the action of $E_{ij}$ on $V$ maps $E_{jj}V$ isomorphically (as an $A$-module) onto $E_{ii}V$. Hence as an $\Mat_n(A) = A\otimes_{\mathbb{C}}\Mat_n(\mathbb{C})$
module, V is isomorphic to $(E_{11}V)\otimes_{\mathbb{C}}\mathbb{C}^n$ --- where $\mathbb{C}^n$ is the standard representation of $\Mat_n(\mathbb{C})$ i.e. is the unique 
irreducible representation of $\Mat_n(\mathbb{C})$.
\[
V\cong (E_{11}V)\otimes_{\mathbb{C}}\mathbb{C}^n
\]
$E_{11}V$ is an irreducible  $A$-module since if not $V=(E_{11}V)\otimes_{\mathbb{C}}\mathbb{C}^n$ would not be an irreducible $A\otimes_{\mathbb{C}}\Mat_n(\mathbb{C})$-module.\\
If $J$ is the null space (in $A$) of $E_{11}V$, then $J^{\Diamond} = J\otimes_{\mathbb{C}}\Mat_n(\mathbb{C})$ is the null space of $V= (E_{11}V)\otimes_{\mathbb{C}}\mathbb{C}^n$
and this completes the proof.      
\end{proof}

\section{Permanence properties}

\begin{theorem}   Denote by $\Gamma$ a finite group.   Stratified equivalence persists under the formation of tensor products 
$A \otimes B$,  and the 
formation of crossed products $A \rtimes \Gamma$.
\end{theorem}

\begin{proposition}Given $k_1$-algebras $A \sim A'$ and $k_2$-algebras $B \sim B'$ then we have
\[
A \otimes B \sim A' \otimes B'
\]
as $k_1 \otimes k_2$-algebras, all tensor products over $\C$.     
\end{proposition} 
\begin{proof}  We first prove that $A \otimes B \sim A' \otimes B$.   Let $\{I_j\}$ be the filtration for $A$, and $\{I'_j\}$ be the filtration for $A'$.
Then $\{I_j \otimes B\}$ (resp. $\{I'_j \otimes B\}$ will serve as filtrations for $A \otimes B$, (resp. $A' \otimes B$), in the sense that the maps
\[
I_j/I_{j-1} \otimes B \to I'_j/I'_{j-1} \otimes B
\]
are spectrum-preserving.   A similar argument will prove that $A' \otimes B \sim A' \otimes B'$.   Then we have
\[
A \otimes B \sim A' \otimes B \sim A' \otimes B'
\]
as required.

Let $\Psi$ be a variation of $k_1$-structure for $A$:
\[
k_1 \to Z(A[t,t^{-1}]) 
\]
and let
\[
\Psi_B : k_2 \to Z(B)
\]
be the $k_2$-structure for $B$.   Then 
\[
\Psi \otimes \Psi_B : k_1 \otimes k_2 \to Z(A[t,t^{-1}]) \otimes Z(B) \simeq Z(A \otimes B)[t,t^{-1}]
\]
will be a variation of $k_1 \otimes k_2$-structure for $A \otimes B$.  
\end{proof}

\begin{lemma}\label{phiGamma}  Assume that $A$ and $B$ both have $\Gamma$ acting (as automorphisms of $k$-algebras), and we are given a 
$\Gamma$-equivariant morphism $\phi: A \to B$ which is spectrum preserving.   Then $\phi \rtimes \Gamma:
A \rtimes \Gamma \to B \rtimes \Gamma$ is a spectrum preserving morphism of $k$-algebras.  
\end{lemma}

\begin{proof}   According to Lemma A.1 in \cite{ABPS}, $\Irr(A \rtimes \Gamma), \Irr(B \rtimes \Gamma)$ are in bijection with the (twisted) extended quotients: 
\[
(\Irr(A) \q \Gamma)_{\natural} \simeq \Irr(A \rtimes \Gamma)
\]
\[ 
 (\Irr(B) \q \Gamma)_{\natural} \simeq \Irr(B \rtimes \Gamma)
\]

Since $\phi : A \to B$ is spectrum preserving, the $2$-cocycles $\natural$ of $\Gamma$ associated to $\Irr(A)$ and $\Irr(B)$ can be identified.   Hence 
$\phi$ determines a bijection of the (twisted) extended quotients 
\[
(\Irr(A) \q \Gamma)_{\natural} \longleftarrow (\Irr(B) \q \Gamma)_{\natural} 
\]
Applying Lemma A.1 of \cite{ABPS}, this implies that $\phi \rtimes \Gamma$ is spectrum preserving.

\end{proof}

\begin{proposition}\label{crossed} If $A$ and $B$ both have $\Gamma$ acting (as automorphisms of $k$-algebras), and we are given a 
$\Gamma$-equivariant stratified equivalence $A \sim B$, then we have 
\[
A \rtimes \Gamma \sim B \rtimes \Gamma
\]
\end{proposition}

\begin{proof}  Suppose now that we are given a $\Gamma$-equivariant morphism of $k$-algebras $\phi : A \to B$ which is, in a $\Gamma$-equivariant way,
spectrum preserving with respect to filtrations, i.e. the filtrations of $A$ and $B$ are preserved by the action of $\Gamma$.  This gives a map of $k$-algebras 
$\phi \rtimes \Gamma : A \rtimes \Gamma$ to $B \rtimes \Gamma$.   We claim that this map is spectrum-preserving with respect to filtrations.   

  Using the filtrations,
consider the map $I_j/I_{j-1} \to J_j / J_{j-1}$. This is a $\Gamma$-equivariant spectrum preserving map.   Hence by Lemma \ref{phiGamma} above, the map
\[
(I_j/I_{j-1}) \rtimes \Gamma \to (J_j / J_{j-1} )\rtimes \Gamma
\]
is spectrum preserving.   So the filtrations of $A \rtimes \Gamma$ and $B \rtimes \Gamma$ given by $I_j\rtimes \Gamma$ and
$J_j \rtimes \Gamma$ are the required filtrations.
\medskip

Suppose given a $\Gamma$-equivariant variation of $k$-structure for $A$, i.e. a $\Gamma$-equivariant unital map
of $\C$-algebras
\[
k \longrightarrow Z(A[t,t^{-1}])
\]
where the $\Gamma$-action on $k$ is trivial.   Due to the triviality of this action, $k$ is mapped to 
\[
Z(A[t,t^{-1}])^{\Gamma} = (Z(A))^\Gamma [t,t^{-1}]
\]
By the standard inclusion 
\[
(Z(A))^\Gamma \subset Z(A \rtimes \Gamma)
\]
$k$ is mapped to $Z(A \rtimes \Gamma)[t,t^{-1}]$:  
\[
k \longrightarrow Z(A \rtimes \Gamma)[t, t^{-1}]
\]
which is  the required variation of $k$-structure for the crossed product $A \rtimes \Gamma$.  
\end{proof}

\section{Affine Hecke Algebras}

 Let $(X,Y,R,R^\vee)$ be a root 
datum in the standard sense \cite[p.73]{L}.   This root datum delivers the following items: 

-- a finite Weyl group $W_0$

--  an extended affine Weyl group $W_0X: = W_0 \ltimes X$

-- for each $q \in \C^\times$, an affine Hecke algebra $\cH_q$

-- a complex torus $T: = \Hom(X, \C^\times)$   

-- a complex variety $X: = T/W_0$

-- a canonical isomorphism   $\cO(X) \simeq Z(\cH_q)$
\medskip

The detailed construction of $\cH_q$ is described in the article of Lusztig \cite[p.74]{L}.   Set $k: = \cO(X)$.   Then, for all $q \in\mathbb{C}^{\times}$, $\mathcal{H}_q$ is a unital finite type $k$-algebra.   
If $q = 1$, then $\cH_1$ is the group algebra of the extended affine Weyl group:
\[
\cH_1 = \C[W_0X]
\]

\subsection{Stratified equivalence for affine Hecke algebras} 

\begin{theorem}\label{roots} Except for $q$ in a finite set of roots of unity,
none of which is 1,  $\mathcal{H}_q$ is stratified equivalent to $\mathcal{H}_1$ :
\[
\mathcal{H}_q \sim\mathcal{H}_1.
\]
\end{theorem}  

\begin{proof}  
 Let $J$ be Lusztig's asymptotic algebra \cite[\S 8]{L}.
As a
$\mathbb{C}$-vector space, $J$ has a basis
$\{t_v\,:\,v \in W_0X \}$, and there is a canonical structure of
associative $\mathbb{C}$-algebra on $J$.  The algebra $\mathcal{H}_q$ 
is viewed as a $k$-algebra via the canonical isomorphism
\[
\mathcal{O}(T/W_0)\cong Z(\mathcal{H}_q) .
\]
Lusztig's map $\phi_q$ maps $Z(\mathcal{H}_q)$ to $Z(J)$ by Proposition 1 in \cite{ABP}
and thus determines a unique $k$-structure for $J$ such that the map
$\phi_q$ is a morphism of $k$-algebras. $J$ with this $k$-structure 
will be denoted $J_q$.

Let 
\[
a : W_0X \to \mathbb{N} \cup \{0\}
\]
 be the weight function defined by Lusztig. 

The algebra $J$ has a unit element of the form $1 = \sum_{d \in \mathcal{D}}$ where $\mathcal{D}$ is a certain set of involutions in $W_a$.   
For any $q \in \C^\times$, the $\C$-linear map $\psi_q : \mathcal{H}_q \to J$ defined by
\begin{eqnarray}\label{Lusztig}
\phi_q(C_w) = \sum h_{w,d,v} t_v
\end{eqnarray}
is a $\C$-linear homomorphism preserving $1$.    The summation is over the set
\begin{eqnarray*}
\{d \in \mathcal{D}, v \in W_0X : a(v) = a(d)\}.
\end{eqnarray*}

The elements $C_w$ are defined in \cite[p.82]{L} and  form a $\C$-basis of $\mathcal{H}_q$.   The coefficients $h_{w,d,v}$ originate in the multiplication rule for
the $C_w$, see \cite[p.82]{L}.

The map $\phi_q$ is injective. Thus all algebras $\mathcal{H}_q$ with $q \in \C^\times$ appear as subalgebras of a single $\C$-algebra  $J$.

 Define
 \begin{eqnarray*}
 J_k  &  = & \textrm{ideal generated  by} \, \{t_v \in J : a(v) \leq k\}\\ 
  I_k & = & \phi^{-1}_q(J_k)
 \end{eqnarray*}
 
 We have the filtrations
 \begin{eqnarray}\label{Filter}
 0 \subset I_1 \subset I_2 \subset \cdots \subset I_r & = & \mathcal{H}_q
 \end{eqnarray}
 \begin{eqnarray*}
  0 \subset J_1 \subset J_2 \subset \cdots \subset J_r &  = & J
 \end{eqnarray*}
 We have
 \begin{eqnarray*}
 \Irr(J) & = & \bigsqcup \Irr(J_k/J_{k-1})\\
 \Irr(\mathcal{H}_q) & = & \bigsqcup \Irr(I_k/I_{k-1})
 \end{eqnarray*}
 
 Let $M$ be a simple $\mathcal{H}_q$-module (resp. simple $J$-module).   The \emph{weight} $a_M$ of $M$ is defined by Lusztig in \cite[p.82]{L}.
 Let $\Irr(\mathcal{H}_q)_k$ (resp. $\Irr(J)_k$) denote the set of simple $\mathcal{H}_q$-modules (resp. simple $J$-modules) of weight $k$.   
 It follows from Lusztig's definition that we have
 \begin{eqnarray*}
 \Irr(J_k/J_{k-1}) &  = &  \Irr(J)_k\\
 \Irr(I_k/I_{k-1}) & = & \Irr(\mathcal{H}_q)_k
 \end{eqnarray*}
 
 \bigskip

 Consider the commutative diagram
 \[
 \begin{CD}
 \mathcal{H}_q  @ >\phi_q>> J\\
 @AAA              @AAA\\
 I_k @ >\phi_q|I_k>> J_k\\
 @VVV @VVV\\
 I_k/I_{k-1} @>>(\phi_q)_k> J_k/J_{k-1} 
 \end{CD}
 \]
 
 \medskip
 
  The middle horizontal map is the restriction of $\phi_q$ to $I_k$.   This map is 
  defined by equation (\ref{Lusztig}) but with summation over the set
  \begin{eqnarray*}
\{d \in \mathcal{D}, v \in W_0X : a(v) = a(d) \leq k\}.
\end{eqnarray*}

The bottom horizontal map is middle horizontal map after descent to the quotient $I_k/I_{k-1}$.   This map is 
  defined by equation (\ref{Lusztig}) but with summation over the set
  \begin{eqnarray*}
\{d \in \mathcal{D}, v \in W_0X : a(v) = a(d) = k\}.
\end{eqnarray*}

We now apply one of Lusztig's main theorems, Theorem 8.1 in \cite{L}:  Assume that $q \in \C^\times$ is either $1$ or is not a root of $1$. 
Let $M'$ be a simple $J$-module of weight $a_{M'}$.   The pre-image $\phi^{-1}_q(M')$ will contain a 
unique subquotient $M$ of maximal weight $a_M = a_{M'}$.   Then the map $M' \mapsto M$ is a bijection from the set of simple $J$-modules (up to isomorphism) to the set of simple
$\mathcal{H}_q$-modules (up to isomorphism).     

We reconcile our construction with Lusztig's bijection by observing that our map 
\[
(\phi_q)_k :  I_k/I_{k-1} \to J_k/J_{k-1}
\]
determines a map 
\begin{eqnarray}\label{weight}
\Irr(J)_k \to \Irr(\mathcal{H}_q)_k
\end{eqnarray}
This map is weight-preserving, by construction,  and coincides with the Lusztig bijection.   The map (\ref{weight}) is spectrum-preserving by Lusztig's theorem, 
and so the original map $\phi_q$ is spectrum preserving with 
respect to the  filtrations (\ref{Filter}).

  \bigskip

 $\mathcal{H}_q$ is then stratified 
equivalent to $\mathcal{H}_1$ by the three elementary steps
\[
\mathcal{H}_q  \rightsquigarrow J_q \rightsquigarrow J_1 \rightsquigarrow \mathcal{H}_1 .
\]
The second elementary step (i.e. passing from $J_q$ to $J_1$) is an algebraic 
variation of $k$-structure with parameter space $\mathbb{C}^{\times}$. The first elementary step uses Lusztig's map $\phi_q$, and the 
third elementary step uses Lusztig's map $\phi_1$.
Hence (provided $q$ is not in the exceptional set
of roots of unity---none of which is 1) $\mathcal{H}_q$ is stratified  equivalent to 
\[
\mathcal{H}_1
= \C [W_0X]=\mathcal{O}(T)\rtimes W_0
\]
\end{proof}

\begin{remark}   In Theorem \ref{roots}, the condition on $q$ can be replaced by the following more precise condition:
\[
\sum _{w \in W_0} q^{\ell(w)} \neq 0
\]
where $\ell: W_0X \to  \mathbb{N}$ is the length function, see \cite[Theorem 3.2]{Xi2}.    
\end{remark} 

 If $q \neq 1$ then $\mathcal{H}_q$ and $\mathcal{H}_1$ are not Morita equivalent as $k$-algebras, for the exceptional sets are not equal:
\[
q \neq 1 \implies \fE(\cH_q) \neq \fE(\cH_1)
\]

\subsection{The affine Hecke algebras of $\SL_2(\C)$}  If the root datum 
\[
(X,Y,R,R^\vee)
\]
 arises from a connected complex reductive Lie group $G$, then we will write
$\cH_q(G) = \cH_q$.   

An interesting situation arises for $\cH_q(\SL_2(\C))$.   
Let $s,t$ be the simple reflections in the (extended) affine Weyl group $W_0X$.   When $q + 1 \neq 0$ there is an 
 isomorphism of $\C$-algebras between $\cH_q$ and $\cH_1$:
\[
\cH_q \simeq \cH_1
\]
\[ 
T_s \mapsto \frac{q+1}{2} s + \frac{q-1}{2}, \quad \quad T_t \mapsto \frac{q+1}{2} t + \frac{q-1}{2}
\]

On the other hand, when $q \neq 1$, $\cH_q$ and $\cH_1$ cannot be isomorphic as $k$-algebras, for the exceptional sets are different.  
Let $T$ denote the standard maximal torus in $\SL_2(\C)$:
\[
\left\{
\left(
\begin{array}{cc}
z & 0 \\
0 & 1/z
\end{array}
\right)
: z \neq 0
\right\}
\]   
We have $W_0 = \Z/2\Z$,  $k = \cO(T/W_0)$  the algebra of $\Z/2\Z$-invariant Laurent polynomials in one variable.   Note that the 
non-trivial element of $W$ acts by 
\[
\left(
\begin{array}{cc}
z & 0 \\
0 & 1/z
\end{array}
\right)
\mapsto 
\left(
\begin{array}{cc}
1/z & 0 \\
0 & z
\end{array}
\right)
\]

The quotient 
variety $T/W_0$ comprises unordered pairs $\{z_1, z_2\}$ of nonzero complex numbers which satisfy the equation
 $z_1z_2 = 1$.   

The exceptional sets are  
\begin{eqnarray*}
\fE(\cH_1) & = & \{- 1, - 1\}  \sqcup \{1, 1\} \\
\fE(\cH_q) & = & \{- 1, -1\} \sqcup \{q, 1/q\}  
\end{eqnarray*}

As $q$ with $q \neq 1$ is deformed to $1$, the point $\{-1, -1\}$ stays fixed, and the point $\{q, 1/q\}$ moves to the point $\{1,1\}$.

 \subsection{The affine Hecke algebras of $\SL_3(\C)$}   The case of $\SL_3(\C)$ is considered in \S 11.7 of Xi \cite{Xi} where it is proved that
 $\cH_q$ and $\cH_1$ are not isomorphic as $\C$-algebras whenever $q \neq 1$.   As $k$-algebras, $\cH_q$ is not Morita equivalent to $\cH_1$
 whenever $q \neq 1$ because the exceptional set $\fE(\cH_q)$ with $q \neq 1$ is not the same as $\fE(\cH_1)$.  
 
  In contrast to this, except for a finite set of roots of unity (none of which is $1$), according to Theorem \ref{roots}, $\cH_q$ is stratified equivalent to $\cH_1$.

\end{document}